\documentclass[11pt,twoside,a4paper]{article} 
\usepackage[T1]{fontenc}
\usepackage[utf8]{inputenc}

\usepackage[australian]{babel}

\usepackage[driver=pdftex,margin=3cm,heightrounded=true,centering,includeheadfoot]{geometry}

\usepackage{hyperref}

\usepackage[intlimits,tbtags]{mathtools} 
\usepackage{amssymb,amsfonts}
\usepackage{amsthm}
\usepackage[all]{xy}
\usepackage[capitalise]{cleveref}

\numberwithin{equation}{section}

\usepackage[dvipsnames]{xcolor}
\definecolor{MyBlue}{cmyk}{1,0.13,0,0.63}
\definecolor{MyGreen}{cmyk}{0.91,0,0.88,0.52}
\definecolor{MyRed}{rgb}{.6,0,0}
\newcommand{\mylinkcolor}{MyBlue}
\newcommand{\mycitecolor}{MyGreen}
\newcommand{\myurlcolor}{MyRed}

\theoremstyle{plain}
\newtheorem{thm}{Theorem}[section]
\newtheorem{lem}[thm]{Lemma}
\newtheorem{prop}[thm]{Proposition}
\newtheorem{coro}[thm]{Corollary}
\theoremstyle{definition}
\newtheorem{defn}[thm]{Definition}
\newtheorem{remark}[thm]{Remark}

\newtheorem{assumption}[thm]{Assumption}
\newtheorem{setting}[thm]{Setting}

\renewcommand{\eqref}[1]{\labelcref{#1}}
\crefname{thm}{Theorem}{Theorems}
\crefname{lem}{Lemma}{Lemmas}
\crefname{prop}{Proposition}{Propositions}
\crefname{coro}{Corollary}{Corollaries}
\crefname{defn}{Definition}{Definitions}
\crefname{example}{Example}{Examples}
\crefname{remark}{Remark}{Remarks}
\crefname{assumption}{Assumption}{Assumptions}

\hypersetup{%
  bookmarksnumbered=true,bookmarksopen=false,%
  plainpages=false,
  linktocpage=true,%
  colorlinks=true,breaklinks=true,%
  linkcolor=\mylinkcolor,citecolor=\mycitecolor,urlcolor=\myurlcolor,%
  pdfpagelayout=OneColumn,%
  pageanchor=true,%
}

\makeatletter
\def\@endtheorem{\endtrivlist}
\def\thm@space@setup{%
  \thm@preskip=4pt plus 2pt minus 2pt
  \thm@postskip=\thm@preskip
}
\renewenvironment{proof}[1][\proofname]{\par
  \pushQED{\qed}%
  \normalfont \topsep4\p@\relax 
  \trivlist
  \item[\hskip\labelsep
        \itshape
    #1\@addpunct{.}]\ignorespaces
}{%
  \popQED\endtrivlist\@endpefalse
}
\makeatother

\usepackage[shortlabels]{enumitem}
\setlist{topsep=4pt plus 2pt minus 2pt,partopsep=0pt,itemsep=2pt plus 2pt minus 2pt,parsep=0.5\parskip}
\setenumerate[1]{label=(\arabic*),font=\upshape}

\usepackage{fancyhdr}
\newcommand{\MR}[1]{}

\let\OLDthebibliography\thebibliography
\renewcommand\thebibliography[1]{
  \addcontentsline{toc}{section}{\refname}
  \OLDthebibliography{#1}
  \setlength{\parskip}{0pt}
  \setlength{\itemsep}{0pt plus 0.3ex}
}

\newcommand{\N}{\mathbb{N}}
\newcommand{\R}{\mathbb{R}}
\newcommand{\C}{\mathbb{C}}
\newcommand{\Z}{\mathbb{Z}}
\newcommand{\A}{\mathcal{A}}
\newcommand{\mH}{\mathcal{H}}
\newcommand{\mK}{\mathcal{K}}
\newcommand{\mL}{\mathcal{L}}
\newcommand{\D}{\mathcal{D}}
\newcommand{\B}{\mathcal{B}}

\DeclareMathOperator{\Dom}{Dom}
\DeclareMathOperator{\Ran}{Ran}
\DeclareMathOperator{\Ker}{Ker}
\DeclareMathOperator{\supp}{supp}
\DeclareMathOperator{\Index}{Index}
\DeclareMathOperator{\ind}{ind}
\DeclareMathOperator{\relind}{rel-ind}
\DeclareMathOperator{\Sig}{Sig}
\DeclareMathOperator{\SF}{sf}

\newcommand{\ev}{{\mathrm{ev}}}
\newcommand{\od}{{\mathrm{od}}}
\newcommand{\op}{{\mathrm{op}}}
\newcommand*{\K}{\relax\ifmmode{K}\else{\itshape K}\fi}
\newcommand*{\KK}{\relax\ifmmode{K\!K}\else{\itshape K\hspace{-.07em}K}\fi}
\newcommand{\into}{\hookrightarrow}
\newcommand{\sfarrow}{\rightarrow}
\newcommand{\til}{\widetilde}

\newcommand{\mattwo}[4]{
  \begin{pmatrix}#1&#2\\ #3&#4\end{pmatrix}
}

\def\MyTitle{A KK-theoretic note on the spectral localiser}
\title{\MyTitle}
\author{
Koen van den Dungen
\\[2mm]
{\small Mathematisches Institut}, 
{\small Universität Bonn}\\
{\small Endenicher Allee 60, D-53115 Bonn}\\
{\small kdungen@uni-bonn.de}
}
\date{}

\begin{document}

\maketitle

\begin{abstract}
\noindent
We review the construction of the spectral localiser (due to Loring and Schulz-Baldes) from a \KK-theoretic perspective. 
We first give a \KK-theoretic argument providing a spectral flow expression for the even or odd index pairing in terms of the ``infinite volume'' spectral localiser. 
Our approach towards this first step is more direct, treats the even and odd cases on an equal footing, and has the advantage that the construction of the spectral localiser becomes immediately apparent from the computation of the index pairing via a Kasparov product. 
In a second step of ``spectral truncation'', we then describe how this spectral flow expression can be computed in terms of the signature of the ``finite volume'' spectral localiser. 
Throughout, we do not require invertibility of the operator representing the \K-homology class, and the even index pairing then obtains an additional contribution coming from the Fredholm index. 

\vspace{\baselineskip}
\K-theory, 
index theory, 
spectral flow. 

\end{abstract}


\section{Introduction}

Spectral localisers were recently introduced by Loring and Schulz-Baldes \cite{LS17,LS20} as a method for computing an even ($j=0$) or odd ($j=1$) index pairing 
\begin{equation}
\label{eq:index_pairing}
\K_j(A) \times \KK^j(A,\C) \longrightarrow \K_0(\C) \simeq \Z ,
\end{equation}
where $A$ is a unital $C^*$-algebra. Here the even or odd $\K$-homology class in $\KK^j(A,\C)$ is represented by a `generalised Dirac operator' $\D$, and the \K-theory class in $\K_j(A)$ is represented in the even case ($j=0$) by the positive spectral projection $P_{>0}(H)$ of a self-adjoint invertible $H\in M_n(A)$ or in the odd case ($j=1$) by an invertible $G\in M_n(A)$. 
The image of such an index pairing $[P_{>0}(H)] \otimes_A [\D]$ or $[G] \otimes_A [\D]$ can be described as a Fredholm index on an infinite-dimensional Hilbert space. 

Such index pairings are not only of purely mathematical interest, but they also appear for instance in condensed matter physics as topological invariants in the theory of (disordered) topological insulators \cite{BES94,Prodan--Schulz-Baldes16}. In this case, the value of the index pairing corresponds to an experimentally measurable quantity. 
Loring and Schulz-Baldes introduced the spectral localiser as a powerful method for computing such quantities. Indeed, the main purpose of the spectral localiser is to reduce the index pairing to the computation of a \emph{finite} matrix. 
As such, this breakthrough makes the computation of the index invariant in \K-theory accessible to numerical calculations and has already seen various applications (see e.g.\ \cite{SS21,SS22,CL22,FG24,CL24,Sto25}). 

The main results on the (even or odd) spectral localiser by Loring and Schulz-Baldes have been reproven in \cite{LS19,LSS19} using a spectral flow approach (see also the book \cite[Ch.10]{DSW2023}). This spectral flow approach has also been adapted to the context of semifinite spectral triples \cite{SS21}. 
Very recently, Li--Mesland \cite{LM25pre} and Kaad \cite{Kaa25pre} have studied spectral localisers from the perspective of unbounded \KK-theory, considering a generalised index pairing $\K_j(A) \times \KK^j(A,B) \to \K_0(B)$, where $B$ is another $C^*$-algebra. 
Li--Mesland \cite{LM25pre} use E-theory to describe the \emph{odd} index pairing \eqref{eq:index_pairing} (with $j=1$) in terms of the asymptotic morphism $\K_1(A) \simeq \K_0(C_0(\R,A)) \to \K_0(\C)$ associated to the \K-homology class in $\KK^1(A,\C)$. 
Kaad \cite{Kaa25pre} focuses on generalising the \emph{even} index pairing \eqref{eq:index_pairing} (with $j=0$) to the bivariant setting $\K_0(A) \times \KK^0(A,B) \to \K_0(B)$. 

In this article, our main goal is to show that the (even or odd) spectral localiser appears very naturally from the \KK-theoretic perspective. 
As in \cite{LM25pre}, we start in \cref{sec:index_pairing} by employing the Bott periodicity isomorphism $\K_j(A) \simeq \K_{j+1}(C_0(\R,A))$, but we represent the resulting \K-theory class in $\K_{j+1}(C_0(\R,A))$ by a Fredholm operator on a Hilbert $C_0(\R,A)$-module (instead of as a formal difference of projections, 
as in \cite[Lemma 3.1]{LM25pre}). Our Fredholm representative then allows us to obtain the (even or odd) index pairing \eqref{eq:index_pairing} via a relatively straightforward computation of the pairing 
\begin{equation}
\label{eq:index_Kasp_prod}
\K_{j+1}(C_0(\R,A)) \times \KK^j(A,\C) \longrightarrow \K_1(C_0(\R)) \simeq \K_0(\C) \simeq \Z 
\end{equation}
as a Kasparov product in the unbounded picture of \KK-theory. 
The result is described as the spectral flow of a certain path of self-adjoint Fredholm operators, where the (self-adjoint invertible) spectral localiser appears as one of the endpoints of this path (the other endpoint is also a spectral localiser, but comes from a trivial \K-theory class). 
More precisely, for any $\kappa>0$, the even or odd (infinite volume) spectral localiser corresponding to the pair $(H,\D)$ in the even case or $(G,\D)$ in the odd case is defined by 
\begin{align*}
L^\ev_\kappa(H,\D) &:= \mattwo{H}{\kappa\D_-}{\kappa\D_+}{-H} , & 
L^\od_\kappa(G,\D) &:= \mattwo{\kappa\D}{G}{G^*}{-\kappa\D} ,
\end{align*}
where in the even case we have decomposed the $\Z_2$-graded operator $\D = \mattwo{0}{\D_-}{\D_+}{0}$. 
We then prove in \cref{thm:even_index_pairing_spec_loc} that the even index pairing can be computed (for sufficiently small $\kappa$) as 
\begin{align}
\label{eq:even_sf_exp}
[P_{>0}(H)] \otimes_A [\D] 
= \SF\big( L^\ev_\kappa(-1,\D) \sfarrow L^\ev_\kappa(H,\D) \big) 
\in \K_0(\C) \simeq \Z ,
\end{align}
where $\SF( \cdot \sfarrow \cdot )$ denotes the spectral flow along the straight line path. 
Similarly, we prove in \cref{thm:odd_index_pairing_spec_loc} that the odd index pairing can be computed (for sufficiently small $\kappa$) as 
\begin{align}
\label{eq:odd_sf_exp}
[G] \otimes_A [\D] 
= \SF\big( L^\od_\kappa(1,\D) \sfarrow L^\od_\kappa(G,\D) \big) 
\in \K_0(\C) \simeq \Z .
\end{align}
Compared to the existing literature on spectral localisers, our approach is more direct in the sense that it does not require any intermediate homotopies (though it is perhaps less elementary, as it does require the full machinery of Kasparov's \KK-theory). 
The particular advantage of our approach is that the construction of the spectral localiser becomes immediately apparent from the computation of the Kasparov product \eqref{eq:index_Kasp_prod}. This shows that the spectral localiser is indeed very natural from the \KK-theoretic perspective. 
Moreover, our approach applies to both the even and the odd index pairing, treating both cases on an equal footing. 

Next, in \cref{sec:spectral_truncation}, we describe the \emph{spectral truncation} of the spectral localiser. The \emph{truncated} spectral localiser $L^\ev_{\kappa,\rho}(H,\D)$ or $L^\od_{\kappa,\rho}(G,\D)$ is a finite matrix, which is obtained by compressing the spectral localiser to a finite-dimensional spectral subspace (specified by the parameter $\rho$) of the operator $\D$ representing the \K-homology class $[\D] \in \KK^j(A,\C)$. Loring and Schulz-Baldes \cite{LS17,LS20} have shown that (for suitable $\kappa,\rho$) this truncated spectral localiser is \emph{invertible}. 
Combining this fact with our spectral flow expression \eqref{eq:even_sf_exp} or \eqref{eq:odd_sf_exp}, we can express the (even or odd) index pairing in terms of the \emph{signature} of the truncated spectral localiser. In contrast to the work of Loring and Schulz-Baldes, we do not require that $\D$ is invertible. 
We will prove in \cref{thm:even_pairing_spec_loc} that the expression for the even index pairing then receives an additional term given by the Fredholm index of $\D$:
\[
[P_{>0}(H)] \otimes_A [\D] 
= \frac12 \Sig\big( L^\ev_{\kappa,\rho}(H,\D) \big) + \frac12 \Index(\D_+) .
\]
If $\D$ is invertible, then $\Index(\D_+) = 0$ and we recover the main result from \cite{LS20}. 
For the odd index pairing, there is no such index correction, and we will prove in \cref{thm:odd_pairing_spec_loc} that 
\[
[G] \otimes_A [\D] 
= \frac12 \Sig\big( L^\od_{\kappa,\rho}(G,\D) \big) .
\]
This recovers the main result of \cite{LS17}, without assuming $\D$ to be invertible.

For the convenience of the reader, we briefly recall in \cref{app:Fredholm} some Fredholm theory for operators on Hilbert $C^*$-modules over a $C^*$-algebra $A$, and in particular describe the (even or odd) relative index and the (even or odd) spectral flow, as defined in a general setting for operators on Hilbert $A$-modules. 
Finally, \cref{app:Kasp_prod} contains a detailed computation of the Kasparov product providing the pairing \eqref{eq:index_Kasp_prod}. 

Whereas in this paper we only consider the (even and odd) index pairings in \emph{complex} \K-theory, the spectral localiser has also been studied in the context of $\Z_2$-valued index pairings in \emph{real} \K-theory \cite{DS21}. We leave it as an open matter for future work to extend our \KK-theoretic approach to the real setting.

\section{Spectral localisers and index pairings}
\label{sec:index_pairing}

\subsection{The even index pairing}
\label{sec:even_index_pairing}

\begin{setting}
\label{setting:even}
Consider a unital $C^*$-algebra $A$. 
Let $\A \subset A$ be a dense unital $*$-subalgebra, and consider an even spectral triple $(\A,\mH,\D)$ over $A$ representing a \K-homology class $[\D] \in \KK^0(A,\C)$. 
To be precise, this means that we have a $*$-representation $\pi\colon A\to\B(\mH)$ on a Hilbert space $\mH$, and a (densely defined) self-adjoint operator $\D\colon\Dom(\D)\subset\mH\to\mH$ with compact resolvents, such that for all $a\in\A$ we have $\pi(a)\cdot\Dom(\D)\subset\Dom(\D)$ and $[\D,\pi(a)]$ extends to a bounded operator on $\mH$. 
We assume that the representation $\pi$ is \emph{nondegenerate}, i.e., $\pi(A)\cdot\mH$ is dense in $\mH$. 

Furthermore, consider a projection $p\in M_n(\A)$ for some $1\leq n\in\N$, representing a \K-theory class $[p] \in \K_0(A)$. 
We fix an invertible self-adjoint element $H \in M_n(\A)$ whose positive spectral projection $P_{>0}(H)$ equals $p$:
\[
P_{>0}(H) := \frac12\big(1+H|H|^{-1}\big) = p 
\]
(for instance, we could choose $H = 2p-1$, but we allow for non-unitary choices as well). 
We denote the \emph{spectral gap} of $H$ by $g := \|H^{-1}\|^{-1}$. 

Since the spectral triple $(\A,\mH,\D)$ is even, there is a $\Z_2$-grading on $\mH$, i.e.\ a direct sum decomposition $\mH = \mH_+ \oplus \mH_-$ along with a grading operator $\Gamma = 1\oplus(-1)$. 
The operator $\D$ is odd and the representation $\pi$ is even, so we can write 
\begin{align*}
\D &= \mattwo{0}{\D_-}{\D_+}{0} , & 
\pi(a) &= \mattwo{\pi_+(a)}{0}{0}{\pi_-(a)} , \quad\forall a\in A .
\end{align*}
\end{setting}

By extending the representation $\pi = \pi_+ \oplus \pi_-$ to $M_n(A) \to \B(\mH^{\oplus n})$, 
we obtain a projection $\pi(p) \in \B(\mH^{\oplus n})$ as well as projections $\pi_+(p) \in \B(\mH_+^{\oplus n})$ and $\pi_-(p) \in \B(\mH_-^{\oplus n})$. 
It is well-known that the even index pairing of the \K-theory class $[p]$ with the \K-homology class $[\D]$ can be computed as a Fredholm index as follows: 
\[
[p] \otimes_A [\D] = \Index\big( \pi_-(p) \D_+ \pi_+(p) \big) .
\]
Alternatively, the even index pairing may be expressed in terms of the spectral flow, as follows. 
Consider the special case $\mH_+ = \mH_-$ and $\pi_+ = \pi_-$. 
Assuming $\D$ is invertible, we may consider the unitary $F_+ := \D_+ |\D_+|^{-1}$. 
Since $[\pi(H),F_+]$ is compact, we obtain from \cref{coro:sf_index_PuP} the equalities 
\begin{align*}
\Index \big( \pi(p) F_+ \pi(p) \big) 
&= \Index \big( P_{>0}(\pi(H)) F_+ P_{>0}(\pi(H)) \big) 
= \SF\big( \pi(H) \sfarrow F_+ \pi(H) F_+^* \big) .
\end{align*}

\subsubsection{The even spectral localiser}
\label{sec:even_spec_loc}

Our aim here is to provide a different spectral flow expression for the even index pairing in terms of the spectral localiser. 
We will not assume that $\D$ is invertible. 

\begin{defn}
\label{defn:even_spec_loc}
Let $H$ and $\D$ be as in \cref{setting:even}. 
The \emph{even spectral localiser} of the pair $(H,\D)$ with tuning parameter $\kappa\in(0,\infty)$ is defined as the operator 
\[
L^\ev_\kappa \equiv L^\ev_\kappa(H,\D) 
:= \kappa \D + \Gamma \pi(H) 
= \mattwo{\pi_+(H)}{\kappa\D_-}{\kappa\D_+}{-\pi_-(H)} 
\quad \text{on } (\Dom\D)^{\oplus n} \subset \mH^{\oplus n} .
\]
When no confusion arises, we will often suppress $\pi$ from our notation. 
\end{defn}

\begin{lem}
\label{lem:even_spec_loc_inv}
Assume $\kappa < g^2 \big\|[\D,H]\big\|^{-1}$. 
Then the even spectral localiser $L^\ev_\kappa$ is invertible. 
\end{lem}
\begin{proof}
The square can be estimated by 
\begin{align*}
\big(L^\ev_\kappa\big)^2 
&= \big( \kappa \D + \Gamma H \big)^2
= \kappa^2 \D^2 + H^2 + \kappa [\D,H] \Gamma
\geq g^2 - \kappa \big\| [\D,H] \big\| 
> 0 ,
\end{align*}
where in the third step we used $\kappa^2\D^2\geq0$ and $H^2\geq g^2$, and in the last step we used the assumption $\kappa < g^2 \big\|[\D,H]\big\|^{-1}$. 
\end{proof}

\subsubsection{The spectral flow expression}
\label{sec:even_index_pairing_spec_loc}

Our starting point is an alternative representation of the even \K-theory class $[p] \in \K_0(A)$ based on the Bott periodicity isomorphism $\K_0(A) \simeq \K_1\big( C_0(\R)\otimes A \big)$. 
For this purpose, we consider continuous functions $\chi_\pm \colon \R \to [0,1]$ with the following properties:
\begin{align}
\label{eq:chi}
\begin{aligned}
& \supp \chi_+ \subset [0,\infty) \text{ and } \supp \chi_- \subset (-\infty,0], \text{ and }\\
& \chi_+(t) = 1 \text{ for } t\geq1 \text{ and } \chi_-(t) = 1 \text{ for } t\leq-1.
\end{aligned}
\end{align}

The following construction of the operator $S_H$ is analogous to the construction in the proof of \cite[Lemma B.3]{vdD25_DS_ind_sf_25b} (the idea goes back to \cite[Appendix A.1]{Wah09}). 
Recall that the Bott periodicity isomorphism $\K_1\big(C_0(\R,A)\big) \xrightarrow{\simeq} \K_0(A)$ is implemented by taking the Kasparov product with $[-i\partial_t] \in \KK^1\big( C_0(\R) , \C \big)$ (see \cref{app:Fredholm}). 

\begin{lem}
\label{lem:S_H}
Let $\chi_\pm \colon \R \to [0,1]$ be continuous functions satisfying the properties in \cref{eq:chi}. 
Consider the (bounded) self-adjoint Fredholm operator 
\[
S_H := -\chi_- + \chi_+ H 
\quad \text{on} \quad C_0(\R,A)^{\oplus n} ,
\]
representing a class $[S_H] \in \K_1\big(C_0(\R,A)\big)$. 
Then we have the equality 
\[
[p] = [S_H] \otimes_{C_0(\R)} [-i\partial_t] \in \K_0(A) .
\]
Thus the map $[p] \mapsto [S_H]$ implements the natural isomorphism $\K_0(A) \xrightarrow{\simeq} \K_1\big(C_0(\R,A)\big)$. 
\end{lem}
\begin{proof}
From \cref{prop:sf_rel_cpt_family,prop:SF_Kasp_prod} we obtain 
\begin{align*}
[S_H] \otimes_{C_0(\R)} [-i\partial_t] 
&= \SF\big( \{S_H(t)\}_{t\in[-1,1]} \big) 
= \relind \big( P_{>0}(S_H(1)) , P_{>0}(S_H(-1)) \big) \\
&= \relind \big( P_{>0}(H) , P_{>0}(-1) \big) 
= \relind \big( p , 0 \big)
= [p] .
\qedhere
\end{align*}
\end{proof}

Recall the notation $\SF(T_0\sfarrow T_1)$ for the spectral flow of the straight line path from $T_0$ to $T_1$ (see \cref{eq:SF_endpoints}). 
We are now ready to provide a spectral flow expression for the even index pairing in terms of the spectral localiser. The proof relies on the computation of the Kasparov product $[S_H] \otimes_A [\D]$, which is described in \cref{app:Kasp_prod}. 
\begin{thm}
\label{thm:even_index_pairing_spec_loc}
Assume $\kappa < g^2 \big\|[\D,H]\big\|^{-1}$. 
The pairing of $[p] \in \K_0(A)$ with $[\D] \in \KK^0(A,\C)$ is given by 
\[
[p] \otimes_A [\D] 
= \SF\big( L^\ev_\kappa(-1,\D) \sfarrow L^\ev_\kappa(H,\D) \big) 
= \SF\big( \kappa\D-\Gamma \sfarrow \kappa\D+\Gamma H \big)
\in \K_0(\C) \simeq \Z .
\]
Moreover, if $\D$ is invertible, we obtain 
\[
[p] \otimes_A [\D] = \SF\big( \kappa\D \sfarrow L^\ev_\kappa(H,\D) \big) .
\]
\end{thm}
\begin{proof}
Using \cref{lem:S_H} along with the properties of the Kasparov product, we can rewrite 
\begin{align*}
[p] \otimes_A [\D] 
&= \big( [S_H] \otimes_{C_0(\R)} [-i\partial_t] \big) \otimes_A [\D] 
= [S_H] \otimes_{C_0(\R,A)} \big( [-i\partial_t] \otimes [\D] \big) \\
&= [S_H] \otimes_{C_0(\R,A)} \big( [\D] \otimes [-i\partial_t] \big) 
= \big( [S_H] \otimes_A [\D] \big) \otimes_{C_0(\R)} [-i\partial_t] .
\end{align*}
By \cref{coro:10_Kasp_prod_S_D}, the Kasparov product $[S_H] \otimes_A [\D]$ is given by $[\mL^\ev_\kappa] \in \K_1\big( C_0(\R) \big)$, where the operator $\mL^\ev_\kappa$ on the Hilbert $C_0(\R)$-module $C_0(\R,\mH^{\oplus n})$ is given by 
\[
\mL^\ev_\kappa(t) := \kappa \D + \Gamma S_H(t) .
\]
Noting that the isomorphism $\otimes_{C_0(\R)} [-i\partial_t] \colon \K_1\big( C_0(\R) \big) \xrightarrow{\simeq} \K_0(\C) \simeq \Z$ is given by the spectral flow (see \cref{prop:SF_Kasp_prod}), we have 
\[
\big( [S_H] \otimes_A [\D] \big) \otimes_{C_0(\R)} [-i\partial_t]
= [\mL^\ev_\kappa] \otimes_{C_0(\R)} [-i\partial_t] 
= \SF\big( \{\mL^\ev_\kappa\}_{t\in[-1,1]} \big) .
\]
Since $\D$ has compact resolvents and $S_H$ is bounded, $\mL^\ev_\kappa(t)$ is a relatively $\D$-compact perturbation of $\kappa\D$ (for each $t\in[-1,1]$). By \cref{prop:sf_rel_cpt_family}, the spectral flow depends only on the endpoints, and we obtain the first statement:
\[
[p] \otimes_A [\D] 
= \SF\big( \mL^\ev_\kappa(-1) \sfarrow \mL^\ev_\kappa(1) \big) 
= \SF\big( \kappa\D - \Gamma \sfarrow \kappa\D + \Gamma H \big) .
\]
Assuming $\D$ is invertible, the middle point $\mL^\ev_\kappa(0) = \kappa\D$ is also invertible. 
For $t\in[-1,0]$, we compute $\mL^\ev_\kappa(t)^2 = \big( \kappa\D - \chi_-(t) \Gamma \big)^2 = \kappa^2 \D^2 + \chi_-(t)^2 \geq \kappa^2 \D^2$, and we see that $\mL^\ev_\kappa(t)$ is invertible for all $t\in[-1,0]$. 
Hence $\SF\big( \{\mL^\ev_\kappa(t)\}_{t\in[-1,0]} \big) = 0$, and we obtain the second statement: 
\[
[p] \otimes_A [\D] 
= \SF\big( \kappa\D - \Gamma \sfarrow \kappa\D \big) + \SF\big( \kappa\D \sfarrow \kappa\D + \Gamma H \big) 
= \SF\big( \kappa\D \sfarrow \kappa\D + \Gamma H \big) .
\qedhere
\]
\end{proof}

\subsection{The odd index pairing}
\label{sec:odd_index_pairing}

\begin{setting}
\label{setting:odd}
Consider a unital $C^*$-algebra $A$. 
Let $\A \subset A$ be a dense unital $*$-subalgebra, and consider an odd spectral triple $(\A,\mH,\D)$ over $A$ representing a \K-homology class $[\D] \in \KK^1(A,\C)$. 
To be precise, this means that we have a $*$-representation $\pi\colon A\to\B(\mH)$ and a (densely defined) self-adjoint operator $\D\colon\Dom(\D)\subset\mH\to\mH$ with compact resolvents, such that for all $a\in\A$ we have $\pi(a)\cdot\Dom(\D)\subset\Dom(\D)$ and $[\D,\pi(a)]$ extends to a bounded operator on $\mH$. 
We assume that the representation $\pi$ is \emph{nondegenerate}, i.e., $\pi(A)\cdot\mH$ is dense in $\mH$. 

Furthermore, consider a unitary $u\in M_n(\A)$ for some $1\leq n\in\N$, representing a class $[u] \in \K_1(A)$. 
We fix an invertible element $G \in M_n(\A)$ whose \emph{phase} equals $u$:
\[
G |G|^{-1} = u 
\]
(for instance, we could simply choose $G = u$). 
We denote the \emph{spectral gap} of $G$ by $g := \|G^{-1}\|^{-1}$. 
\end{setting}

We obtain a unitary $\pi(u)$ and an invertible $\pi(G)$ in $\B(\mH^{\oplus n})$ by extending the representation $\pi$ to $M_n(A) \to \B(\mH^{\oplus n})$. 
It is well-known that the odd index pairing of the \K-theory class $[u]$ with the \K-homology class $[\D]$ can be computed as a Fredholm index as follows: 
\[
[u] \otimes_A [\D] = \Index\big( P_{>0}(\D) \pi(u) P_{>0}(\D) \big) . 
\]
Using \cref{coro:sf_index_PuP}, we may also express the index pairing as a spectral flow:
\[
[u] \otimes_A [\D] 
= - \SF\big( \D \sfarrow \pi(u^*) \D \pi(u) \big) 
= \SF\big( \D \sfarrow \pi(u) \D \pi(u^*) \big) . 
\]

\subsubsection{The odd spectral localiser}
\label{sec:odd_spec_loc}

Our aim here is to provide a different spectral flow expression for the odd index pairing in terms of the spectral localiser. 
We will not assume that $\D$ is invertible. 

\begin{defn}
\label{defn:odd_spec_loc}
Let $G$ and $\D$ be as in \cref{setting:odd}. 
The \emph{odd spectral localiser} of the pair $(G,\D)$ with tuning parameter $\kappa\in(0,\infty)$ is defined as the operator 
\[
L^\od_\kappa \equiv L^\od_\kappa(G,\D) 
:= \mattwo{\kappa\D}{\pi(G)}{\pi(G^*)}{-\kappa\D} 
\quad \text{on } (\Dom\D)^{\oplus n} \oplus (\Dom\D)^{\oplus n} \subset \mH^{\oplus n} \oplus \mH^{\oplus n} .
\]
When no confusion arises, we will often suppress $\pi$ from our notation. 
\end{defn}

\begin{lem}
\label{lem:odd_spec_loc_inv}
Assume $\kappa < g^2 \big\|[\D,G]\big\|^{-1}$. 
Then the odd spectral localiser $L^\od_\kappa$ is invertible. 
\end{lem}
\begin{proof}
As in the proof of \cref{lem:even_spec_loc_inv}, the square can be estimated by 
\begin{align*}
\big(L^\od_\kappa\big)^2 
&= \mattwo{\kappa^2 \D^2 + GG^*}{\kappa[\D,G]}{-\kappa[\D,G^*]}{\kappa^2 \D^2 + G^*G} 
\geq g^2 - \kappa \big\| [\D,G] \big\| 
> 0 .
\qedhere
\end{align*}
\end{proof}

\subsubsection{The spectral flow expression}
\label{sec:odd_index_pairing_spec_loc}

We closely follow the discussion of the even case given in \cref{sec:even_index_pairing}, and we will therefore be brief regarding some of the details. 
Our starting point is now an alternative representation of the odd \K-theory class $[u] \in \K_1(A)$ based on the Bott periodicity isomorphism $\K_1(A) \simeq \K_0\big( C_0(\R)\otimes A \big)$. 
Recall the continuous functions $\chi_\pm \colon \R \to [0,1]$ with the properties given in \cref{eq:chi}. 
The following construction of the operator $S_G$ can be found in the proof of \cite[Lemma B.3]{vdD25_DS_ind_sf_25b} in the special case $G=u$ (the idea of the construction goes back to \cite[Appendix A.1]{Wah09}). 

\begin{lem}
\label{lem:S_G}
Let $\chi_\pm \colon \R \to [0,1]$ be continuous functions satisfying the properties in \cref{eq:chi}. 
Consider the (bounded) self-adjoint Fredholm operator 
\[
S_G := \mattwo{0}{\chi_- + \chi_+ G}{\chi_- + \chi_+ G^*}{0} 
\quad \text{on} \quad C_0(\R,A)^{\oplus n} \oplus C_0(\R,A)^{\oplus n} ,
\]
representing a class $[S_G] \in \K_0\big(C_0(\R,A)\big)$. 
Then we have the equality 
\[
[u] = - [S_G] \otimes_{C_0(\R)} [-i\partial_t] \in \K_1(A) .
\]
Thus the map $[u] \mapsto -[S_G]$ implements the natural isomorphism $\K_1(A) \xrightarrow{\simeq} \K_0\big(C_0(\R,A)\big)$. 
\end{lem}
\begin{proof}
We first compute 
\begin{align*}
P_{>0}\left( \mattwo{0}{G}{G^*}{0} \right) 
&= \frac12 \left( \mattwo{1}{0}{0}{1} + \mattwo{0}{G}{G^*}{0} \left| \mattwo{0}{G}{G^*}{0} \right|^{-1} \right) \nonumber\\
&= \frac12 \mattwo{1}{G|G|^{-1}}{G^*|G^*|^{-1}}{1} 
= \frac12 \mattwo{1}{u}{u^*}{1} .
\end{align*}
From \cref{prop:sf_rel_cpt_family,prop:SF_Kasp_prod} we then obtain 
\begin{align*}
[S_G] \otimes_{C_0(\R)} [-i\partial_t] 
&= \SF_1\big( \{S_G(t)\}_{t\in[-1,1]} \big) 
= \relind_1 \big( P_{>0}(S_G(1)) , P_{>0}(S_G(-1)) \big) \\
&= \relind_1 \left( \frac12 \mattwo{1}{u}{u^*}{1} , \frac12 \mattwo{1}{1}{1}{1} \right) 
= [u^*] 
= -[u] .
\qedhere
\end{align*}
\end{proof}

\begin{thm}
\label{thm:odd_index_pairing_spec_loc}
Assume $\kappa < g^2 \big\|[\D,G]\big\|^{-1}$. 
The pairing of $[u] \in \K_1(A)$ with $[\D] \in \KK^1(A,\C)$ is given by 
\begin{align*}
[u] \otimes_A [\D] 
&= \SF\big( L^\od_\kappa(1,\D) \sfarrow L^\od_\kappa(G,\D) \big) 
= \SF\left( \mattwo{\kappa\D}{1}{1}{-\kappa\D} \sfarrow \mattwo{\kappa\D}{G}{G^*}{-\kappa\D} \right) \in \Z .
\end{align*}
Moreover, if $\D$ is invertible, we obtain 
\[
[u] \otimes_A [\D] = \SF\left( \mattwo{\kappa\D}{0}{0}{-\kappa\D} \sfarrow \mattwo{\kappa\D}{G}{G^*}{-\kappa\D} \right) .
\]
\end{thm}
\begin{proof}
As in the proof of \cref{thm:even_index_pairing_spec_loc}, we can now use \cref{lem:S_G} along with the properties of the Kasparov product to rewrite 
\begin{align*}
[u] \otimes_A [\D] 
&= - \big( [S_G] \otimes_{C_0(\R)} [-i\partial_t] \big) \otimes_A [\D] 
= - [S_G] \otimes_{C_0(\R,A)} \big( [-i\partial_t] \otimes [\D] \big) \\
&= [S_G] \otimes_{C_0(\R,A)} \big( [\D] \otimes [-i\partial_t] \big) 
= \big( [S_G] \otimes_A [\D] \big) \otimes_{C_0(\R)} [-i\partial_t] . 
\end{align*}
where in the third step we used that the exterior Kasparov product of \emph{odd} Kasparov modules is \emph{anticommutative}. 
By \cref{coro:01_Kasp_prod_S_D}, the Kasparov product $[S_G] \otimes_A [\D]$ is given by $[\mL^\od_\kappa] \in \K_1\big( C_0(\R) \big)$, where the operator $\mL^\od_\kappa$ on the Hilbert $C_0(\R)$-module $C_0(\R,\mH^{\oplus2n})$ is given by 
\(
\mL^\od_\kappa(t) := L^\od(S_G(t),\D) .
\)
Combined with \cref{prop:SF_Kasp_prod}, we obtain 
\[
\big( [S_G] \otimes_A [\D] \big) \otimes_{C_0(\R)} [-i\partial_t]
= [\mL^\od_\kappa] \otimes_{C_0(\R)} [-i\partial_t] 
= \SF\big( \{\mL^\od_\kappa\}_{t\in[-1,1]} \big) .
\]
Since $\mL^\od_\kappa(t)$ is a relatively $\D$-compact perturbation of $(\kappa\D)\oplus(-\kappa\D)$ (for each $t\in[-1,1]$), the spectral flow depends only on the endpoints (see \cref{prop:sf_rel_cpt_family}), and we obtain the first statement:
\[
[u] \otimes_A [\D] 
= \SF\big( L^\od_\kappa(1,\D) \sfarrow L^\od_\kappa(G,\D) \big) .
\]
Assuming $\D$ is invertible, the middle point $\mL^\od_\kappa(0) = (\kappa\D)\oplus(-\kappa\D)$ is also invertible. 
For $t\in[-1,0]$, we compute $\mL^\od_\kappa(t)^2 = \kappa^2 \D^2 + \chi_-(t)^2 \geq \kappa^2 \D^2$, and we see that $\mL^\od_\kappa(t)$ is invertible for all $t\in[-1,0]$. 
Hence $\SF\big( \{\mL^\od_\kappa(t)\}_{t\in[-1,0]} \big) = 0$, and we obtain the second statement: 
\begin{align*}
[u] \otimes_A [\D] 
&= \SF\big( (\kappa\D)\oplus(-\kappa\D) \sfarrow L^\od_\kappa(G,\D) \big) .
\qedhere
\end{align*}
\end{proof}

\begin{remark}
The spectral flow expression of \cref{thm:odd_index_pairing_spec_loc} for invertible $\D$ already appeared in \cite[\S3]{LS19}, 
but it was noted by Loring and Schulz-Baldes in their proof that ``one should, strictly speaking, always replace $\D$ by $F_\rho(\D)$'' (where $F_\rho(\D)$ is a \emph{bounded} operator obtained from $\D$ via continuous functional calculus). 
The proof given here works directly for unbounded operators and does not involve any intermediate homotopies. 
\end{remark}

\section{Spectral truncation of the index pairing}
\label{sec:spectral_truncation}

We consider an (even or odd) spectral triple $(\A,\mH,\D)$. 
For any \emph{spectral radius} $\rho \in (0,\infty)$ we consider the spectral projection $P_{[-\rho,\rho]}(\D)$. 
We introduce the short notation $\mH_\rho := \Ran P_{[-\rho,\rho]}(\D)$ and $\mH_{\rho^c} := (\mH_\rho)^\perp$. 
We denote by $\pi_\rho \colon \mH \to \mH_\rho$ the surjective partial isometry with $\Ker \pi_\rho = \mH_{\rho^c}$, and note that $\pi_\rho^* \colon \mH_\rho \into \mH$ is the natural inclusion. For an operator $T$ on $\mH$, we denote by $T_\rho := \pi_\rho T \pi_\rho^*$ its compression on $\mH_\rho$ (and similarly for $T_{\rho^c}$). 

\subsection{The even index pairing}
\label{sec:spectral_truncation_even}

We consider again an even spectral triple $(\A,\mH,\D)$ over $A$ and a self-adjoint invertible $H \in M_n(\A)$ with spectral gap $g := \|H^{-1}\|^{-1}$ whose positive spectral projection equals $p$, as in \cref{setting:even}. 
Recall the definition of the \emph{even spectral localiser} $L^\ev_\kappa$ of the pair $(H,\D)$ with tuning parameter $\kappa\in(0,\infty)$ from \cref{defn:even_spec_loc}. 

\begin{defn}
\label{defn:even_spec_loc_finite}
The even \emph{truncated} spectral localiser of the pair $(H,\D)$ with tuning parameter $\kappa$ and spectral radius $\rho\in(0,\infty)$ is defined as the operator 
\[
L^\ev_{\kappa,\rho} 
\equiv L^\ev_{\kappa,\rho}(H,\D) 
:= \pi_\rho L^\ev_\kappa(H,\D) \pi_\rho^* 
= \kappa \D_\rho + \Gamma_\rho H_\rho .
\]
\end{defn}

\begin{assumption}
\label{ass:kappa_rho}
We assume that the tuning parameter $\kappa\in(0,\infty)$ and the spectral radius $\rho\in(0,\infty)$ satisfy the following conditions:
\begin{align*}
\kappa &\leq \frac{g^3}{12 \|H\| \|[\D,H]\|} , & 
\frac{2g}{\kappa} &< \rho . 
\end{align*}
\end{assumption}

\begin{thm}[cf.\ {\cite[Theorem 10.3.1]{DSW2023}}]
\label{thm:even_spec_loc_rho_inv}
The even truncated spectral localiser $L^\ev_{\kappa,\rho}$ is invertible with spectral gap $\frac12 g$. 
Furthermore, the signature of $L^\ev_{\kappa,\rho}$ is independent of the choices of $\kappa$ and $\rho$ (satisfying \cref{ass:kappa_rho}). 
\end{thm}

\begin{remark}
The invertibility of the truncated spectral localiser is well-established in the literature; in the even case, it was first proven in \cite{LS20}, and we cite here the result from the textbook \cite{DSW2023}. Although \cite[Theorem 10.3.1]{DSW2023} assumes $\D$ to be invertible, we note that the first part of its proof, which yields the above theorem, does not actually require this assumption. The invertibility of $\D$ is only used in the proof of the second statement in \cite[Theorem 10.3.1]{DSW2023}, which we will reprove in \cref{thm:even_pairing_spec_loc} below without assuming invertibility of $\D$. 
(Alternatively, see the proof of \cite[Proposition 10.1]{Kaa25pre}, which also does not require invertibility of $\D$.)
\end{remark}

\begin{lem}
\label{lem:even_spec_loc_rho_c_inv}
The operator $L^\ev_{\kappa,\rho^c} := \pi_{\rho^c} L^\ev_\kappa \pi_{\rho^c}^*$ is invertible with spectral gap $\sqrt{\frac{47}{48}}\kappa\rho$. 
\end{lem}
\begin{proof}
From \cref{ass:kappa_rho} we obtain the inequalities 
\begin{align*}
\kappa^2 \rho^2 
&> 4 g^2
\geq 4 \frac{12 \kappa \|H\| \|[\D,H]\|}{g} 
\geq 48 \kappa \|[\D,H]\| ,
\end{align*}
where in the last step we used $\|H\| \geq g$. 
The square of $L^\ev_{\kappa,\rho^c}$ can then be estimated by 
\begin{align*}
\big(L^\ev_{\kappa,\rho^c}\big)^2 
&= \big( \kappa \D_{\rho^c} + \Gamma_{\rho^c} H_{\rho^c} \big)^2
= \kappa^2 \D_{\rho^c}^2 + H_{\rho^c}^2 + \kappa [\D_{\rho^c},H_{\rho^c}] \Gamma_{\rho^c} \\
&\geq \kappa^2 \rho^2 - \kappa \big\| [\D,H] \big\| 
> \frac{47}{48} \kappa^2 \rho^2 ,
\end{align*}
where we used $\D_{\rho^c}^2 \geq \rho^2$, $H_{\rho^c}^2 \geq 0$, and $\|[\D_{\rho^c},H_{\rho^c}]\| \leq \|[\D,H]\|$. 
\end{proof}

Next, we need to prove that the spectral flow from $L^\ev_\kappa$ to $L^\ev_{\kappa,\rho} \oplus L^\ev_{\kappa,\rho^c}$ vanishes. 
The following lemma can be extracted from the proofs of \cite[Theorems 10.3.1 \& 10.4.1]{DSW2023}. 
\begin{lem}
\label{lem:inv_matrix}
For $j=1,2$, let $L_j$ be densely defined self-adjoint invertible operators on a Hilbert space $\mH$, with spectral gaps $g_j := \|L_j^{-1}\|^{-1}$. 
If $B \in \B(\mH)$ satisfies $\|B\| < \sqrt{g_1g_2}$, then the operator 
\[
\mattwo{L_1}{B^*}{B}{L_2} 
\quad \text{on } \Dom L_1 \oplus \Dom L_2 \subset \mH \oplus \mH 
\]
is also invertible.
\end{lem}
\begin{proof}
We conjugate by $|L_1|^{-\frac12} \oplus |L_2|^{-\frac12}$ to obtain a new operator 
\[
\mattwo{L_1|L_1|^{-1}}{|L_1|^{-\frac12} B^* |L_2|^{-\frac12}}{|L_2|^{-\frac12} B |L_1|^{-\frac12}}{L_2|L_2|^{-1}} .
\]
Since the operators $L_j |L_j|^{-1}$ are unitary and $\big\| |L_2|^{-\frac12} B |L_1|^{-\frac12} \big\| \leq \sqrt{g_2}^{-1} \|B\| \sqrt{g_1}^{-1} < 1$, this new operator is invertible, and the statement follows. 
\end{proof}

Now let us consider the straight line path 
\[
L(t) := L^\ev_{\kappa,\rho} \oplus L^\ev_{\kappa,\rho^c} + t \pi_\rho \Gamma H \pi_{\rho^c}^* + t \pi_{\rho^c} \Gamma H \pi_\rho^* , \quad t\in[0,1] ,
\]
between $L(0) = L^\ev_{\kappa,\rho} \oplus L^\ev_{\kappa,\rho^c}$ and $L(1) = L^\ev_\kappa$. 

\begin{lem}
\label{lem:even_spec_loc_oplus}
For sufficiently large $\rho$, 
the operator $L(t)$ is invertible for each $t\in[0,1]$, and we have the equality 
\[
\SF\big( \kappa\D-\Gamma \sfarrow L^\ev_\kappa \big) = \SF\big( \kappa\D-\Gamma \sfarrow L^\ev_{\kappa,\rho} \oplus L^\ev_{\kappa,\rho^c} \big) .
\]
\end{lem}
\begin{proof}
Since $\pi_\rho$ and $\pi_{\rho^c}$ are partial isometries, we have 
\[
\left\| t \pi_\rho \Gamma H \pi_{\rho^c}^* + t \pi_{\rho^c} \Gamma H \pi_\rho^* \right\| \leq \|H\| , 
\quad \text{for all } t\in[0,1] .
\]
From \cref{thm:even_spec_loc_rho_inv,lem:even_spec_loc_rho_c_inv} we have $\big| L^\ev_{\kappa,\rho} \big| \geq g_1 := \frac{g}{2}$ and $\big| L^\ev_{\kappa,\rho^c} \big| \geq g_2 := \sqrt{\frac{47}{48}}\kappa\rho$. 
For any fixed $\kappa$, we may choose $\rho$ to be as large as needed to ensure that 
\[
\|H\| < \sqrt{g_1g_2} = \sqrt[4]{\frac{47}{192}} \sqrt{g\kappa\rho} .
\]
It then follows from \cref{lem:inv_matrix} that $L(t)$ is invertible for all $t\in[0,1]$. 
Since $\pi_\rho$ has finite-dimensional range, $L(t)$ is a compact perturbation of $L^\ev_\kappa$, so that the straight line path from $\kappa\D-\Gamma$ to $L(t)$ lies within the Fredholm operators for all $t\in[0,1]$. The homotopy invariance of the spectral flow then yields the second part of the statement. 
\end{proof}

\begin{thm}
\label{thm:even_pairing_spec_loc}
Consider an even spectral triple $(\A,\mH,\D)$ and a projection $p \in M_n(\A)$ as in \cref{setting:even}. 
Let $H \in M_n(\A)$ be invertible and self-adjoint with spectral gap $g = \|H^{-1}\|^{-1}$ such that $P_{>0}(H) = p$. 
Assume the parameters $\kappa$ and $\rho$ satisfy \cref{ass:kappa_rho}. 
Then the even index pairing of $[p] \in \K_0(A)$ with $[\D] \in \KK^0(A,\C)$ is given by 
\[
[p] \otimes_A [\D] 
= \frac12 \Sig\big( L^\ev_{\kappa,\rho}(H,\D) \big) + \frac12 \Index(\D_+) .
\]
\end{thm}
\begin{proof}
By \cref{thm:even_spec_loc_rho_inv}, the signature of $L^\ev_{\kappa,\rho}$ is independent of the choices of $\kappa$ and $\rho$, so we may assume $\rho$ to be as large as necessary. 
From \cref{thm:even_index_pairing_spec_loc,lem:even_spec_loc_oplus} we obtain the direct sum decomposition 
\begin{align*}
[p] \otimes_A [\D] 
&= \SF\big( \kappa\D - \Gamma \sfarrow L^\ev_\kappa \big) 
= \SF\big( \kappa\D_\rho - \Gamma_\rho \sfarrow L^\ev_{\kappa,\rho} \big) + \SF\big( \kappa\D_{\rho^c}  - \Gamma_{\rho^c} \sfarrow L^\ev_{\kappa,\rho^c} \big) .
\end{align*}
The second summand is given by the spectral flow along the straight line path 
\[
\kappa \D_{\rho^c} + \pi_{\rho^c} \Gamma \big( t - 1 + t H \big) \pi_{\rho^c}^* , \quad t\in[0,1] .
\]
Choosing $\rho$ to be sufficiently large such that $\max\big\{1,\|H\|\big\} < \kappa \rho$, it follows from $\big| \kappa \D_{\rho^c} \big| \geq \kappa\rho$ that this path lies within the invertibles. 
Hence we obtain 
\[
[p] \otimes_A [\D] 
= \SF\big( \kappa\D_\rho-\Gamma_\rho \sfarrow L^\ev_{\kappa,\rho} \big) .
\]
Now the Hilbert space $\mH_\rho$ is \emph{finite-dimensional}. The spectral flow can then easily be computed in terms of the signatures of the invertible operators at the endpoints:
\[
\SF\big( \kappa\D_\rho-\Gamma_\rho \sfarrow L^\ev_{\kappa,\rho} \big) 
= \frac12 \Big( \Sig\big( L^\ev_{\kappa,\rho} \big) - \Sig\big( \kappa\D_\rho - \Gamma_\rho \big) \Big) .
\]
It only remains to show that the second term on the right hand side is given by the Fredholm index of $\D_+$. 
For this purpose, we note that $\D$ is odd with respect to the $\Z_2$-grading $\Gamma$ (and $\Gamma$ commutes with $P_{[-\rho,\rho]}(D) = \pi_\rho^*\pi_\rho$), so that $L^\ev_{\kappa,\rho}(-1,\D) = \kappa\D_\rho - \Gamma_\rho$ is invertible for all $\kappa\in[0,\infty)$, and in particular its signature is independent of $\kappa$. We then compute 
\begin{align*}
\Sig\big( L^\ev_{\kappa,\rho}(-1,\D) \big)
&= \Sig\big( -\Gamma_\rho \big) 
= \Sig\big( -\Gamma_{\Ker\D} \big) \\
&= - \dim\Ker(\D_+) + \dim\Ker(\D_-) 
= - \Index(\D_+) .
\qedhere 
\end{align*}
\end{proof}

\begin{remark}
\begin{enumerate}
\item 
If $\D$ is invertible, then $\Index(\D_+)$ vanishes, and we recover the main result from \cite{LS20}. If $\D$ is not invertible, the correction term given by the Fredholm index also appeared in \cite[Theorem 10.2]{Kaa25pre}, albeit in slight disguise (in the form $\tfrac12 \Sig\big( \Gamma_\rho \big)$). 

\item
We have seen in the proof of \cref{thm:even_pairing_spec_loc} that 
\begin{equation}
\label{eq:Gamma_rho_Index}
\Sig\big( L^\ev_{\kappa,\rho}(-1,\D) \big) = - \Index(\D_+) . 
\end{equation}
A similar computation shows that $\Sig\big( L^\ev_{\kappa,\rho}(1,\D) \big) = \Index(\D_+)$. 
Thus, choosing $H=1$, \cref{thm:even_pairing_spec_loc} recovers the fact that $[1] \otimes_A \cdot \colon \KK^0(A,\C) \to \K_0(\C)$ is the index map: 
\[
[1] \otimes_A [\D] = \frac12 \Sig\big( L^\ev_{\kappa,\rho}(1,\D) \big) + \frac12 \Index(\D_+) = \Index(\D_+) .
\]
Similarly, for $H=-1$, \cref{thm:even_pairing_spec_loc} reduces to the obvious equality $[0] \otimes_A [\D] = 0$. 

\item 
Another approach towards the non-invertible case is the following `doubling trick' \cite[\S6]{Con85}, which replaces the non-invertible operator by an invertible one (in the context of spectral localisers, see also \cite[Remark 10.1.2]{DSW2023}). 
Given an even spectral triple $(\A,\mH,\D)$, with representation $\pi \colon \A \to \B(\mH)$, and a non-zero real number $\mu$, we may instead consider the spectral triple $(\A,\til\mH,\til\D_\mu)$, with representation $\til\pi \colon \A \to \B(\til\mH)$, given by 
\begin{align*}
\til\mH &= \mH \oplus \mH^\op , & 
\til\D_\mu &= \mattwo{\D}{\mu}{\mu}{-\D} , & 
\til\pi(a) &= \mattwo{a}{0}{0}{0} , \quad \forall a\in\A ,
\end{align*}
where $\mH^\op$ denotes the Hilbert space $\mH$ equipped with the opposite $\Z_2$-grading. 
This doubling does not affect the \K-homology class of the spectral triple (the off-diagonal $\mu$ is just a bounded perturbation, and the bottom-right $-\D$ adds a degenerate triple) and therefore also does not affect the index pairing (see also \cite[Lemma 2.10 \& Corollary 2.11]{CGRS14}). 
However, the price one pays is that the representation is no longer nondegenerate. 

Let us consider the effect of this doubling on the right hand side of the expression in \cref{thm:even_pairing_spec_loc}. 
Since $\til\D_\mu$ is by construction invertible, its index vanishes. 
In fact, even if we take $\mu=0$, we still have $\Index(\til\D_0) = 0$, because $\Index(D|_{\mH^\op}) = - \Index(D|_{\mH})$. 
To compute the signature of $L^\ev_{\kappa,\rho}(H,\til\D_\mu)$, we may choose $\kappa$ sufficiently small such that the conditions in \cref{ass:kappa_rho} are satisfied for the pair $(H,\til\D_{t\mu})$ for every $t\in[0,1]$. 
Hence it suffices to consider $\mu=0$. In this case, $\til\D_0$ decomposes as a direct sum, and we have 
\begin{align*}
\Sig\big( L^\ev_{\kappa,\rho}(H,\til\D_\mu) \big) 
&= \Sig\big( L^\ev_{\kappa,\rho}(H,\til\D_0) \big) \\
&= \Sig\big( L^\ev_{\kappa,\rho}(H,\D|_{\mH}) \big) + \Sig\big( L^\ev_{\kappa,\rho}(-1,-\D|_{\mH^\op}) \big) .
\end{align*}
As computed above in \cref{eq:Gamma_rho_Index}, the second term on the right hand side equals 
\[
\Sig\big( L^\ev_{\kappa,\rho}(-1,-\D|_{\mH^\op}) \big) = \Index(\D_+) ,
\]
where the minus sign now disappears because of the opposite grading on $\mH^\op$. 
Thus we have 
\[
[p] \otimes_A [\D] 
= [p] \otimes_A [\til\D_\mu] 
= \frac12 \Sig\big( L^\ev_{\kappa,\rho}(H,\til\D_\mu) \big) 
= \frac12 \Sig\big( L^\ev_{\kappa,\rho}(H,\D) \big) + \frac12 \Index(\D_+) .
\]
In conclusion, the index of $\D_+$ remains present if we consider the invertible operator $\til\D_\mu$, but is now hidden in the spectral localiser of the degenerate summand $-\D|_{\mH^\op}$. 
\end{enumerate}
\end{remark}

\subsection{The odd index pairing}

We consider again an odd spectral triple $(\A,\mH,\D)$ over $A$ and an invertible $G \in M_n(\A)$ with spectral gap $g := \|G^{-1}\|^{-1}$ whose phase equals $u$, as in \cref{setting:odd}. 
Recall the definition of the \emph{odd spectral localiser} $L^\od_\kappa$ of the pair $(G,\D)$ with tuning parameter $\kappa\in(0,\infty)$ from \cref{defn:odd_spec_loc}. 

\begin{defn}
\label{defn:odd_spec_loc_finite}
The odd \emph{truncated} spectral localiser of the pair $(G,\D)$ with tuning parameter $\kappa$ and spectral radius $\rho\in(0,\infty)$ is defined as the operator 
\[
L^\od_{\kappa,\rho} 
\equiv L^\od_{\kappa,\rho}(G,\D) 
:= \pi_\rho L^\od_\kappa(G,\D) \pi_\rho^* 
= \mattwo{\kappa \D_\rho}{G_\rho^*}{G_\rho}{-\kappa \D_\rho} .
\]
\end{defn}

We assume that the tuning parameter $\kappa\in(0,\infty)$ and the spectral radius $\rho\in(0,\infty)$ satisfy the conditions from \cref{ass:kappa_rho}. 

\begin{thm}[{\cite[Theorem 10.4.1]{DSW2023}}]
\label{thm:odd_spec_loc_rho_inv}
The odd truncated spectral localiser $L^\od_{\kappa,\rho}$ is invertible with spectral gap $\frac12 g$. 
Furthermore, the signature of $L^\od_{\kappa,\rho}$ is independent of the choices of $\kappa$ and $\rho$ (satisfying \cref{ass:kappa_rho}). 
\end{thm}

\begin{remark}
The invertibility of the truncated spectral localiser is well-established in the literature; it was first proven in \cite{LS17}, and we cite here the result from the textbook \cite{DSW2023}. Although \cite[Theorem 10.4.1]{DSW2023} assumes $\D$ to be invertible, we note that the first part of its proof, which yields the above theorem, does not actually require this assumption. The invertibility of $\D$ is only used in the proof of the second statement in \cite[Theorem 10.4.1]{DSW2023}, which we will reprove in \cref{thm:odd_pairing_spec_loc} below without assuming invertibility of $\D$. 
\end{remark}

\begin{lem}
\label{lem:odd_spec_loc_rho_c_inv}
The operator $L^\od_{\kappa,\rho^c} := \pi_{\rho^c} L^\od_\kappa \pi_{\rho^c}^*$ is invertible with spectral gap $\sqrt{\frac{47}{48}}\kappa\rho$. 
\end{lem}
\begin{proof}
The proof is similar to the proof of \cref{lem:even_spec_loc_rho_c_inv}. 
\end{proof}

We now consider the straight line path 
\[
L(t) := L^\od_{\kappa,\rho} \oplus L^\od_{\kappa,\rho^c} + t \pi_\rho \mattwo{0}{G}{G^*}{0} \pi_{\rho^c}^* + t \pi_{\rho^c} \mattwo{0}{G}{G^*}{0} \pi_\rho^* , \quad t\in[0,1] ,
\]
between $L(0) = L^\od_{\kappa,\rho} \oplus L^\od_{\kappa,\rho^c}$ and $L(1) = L^\od_\kappa$. 
\begin{lem}
\label{lem:odd_spec_loc_oplus}
For sufficiently large $\rho$, the operator $L(t)$ is invertible for each $t\in[0,1]$, and we have the equality 
\[
\SF\big( L^\od_\kappa(1,\D) \sfarrow L^\od_\kappa(G,\D) \big) = \SF\big( L^\od_\kappa(1,\D) \sfarrow L^\od_{\kappa,\rho}(G,\D) \oplus L^\od_{\kappa,\rho^c}(G,\D) \big) .
\]
\end{lem}
\begin{proof}
The proof is similar to the proof of \cref{lem:even_spec_loc_oplus}. 
\end{proof}

\begin{thm}
\label{thm:odd_pairing_spec_loc}
Consider an odd spectral triple $(\A,\mH,\D)$ and a unitary $u \in M_n(\A)$ as in \cref{setting:odd}. 
Let $G \in M_n(\A)$ be invertible with spectral gap $g = \|G^{-1}\|^{-1}$ such that $G|G|^{-1} = u$. 
Assume the parameters $\kappa$ and $\rho$ satisfy \cref{ass:kappa_rho}. 
Then the odd index pairing of $[u] \in \K_1(A)$ with $[\D] \in \KK^1(A,\C)$ is given by 
\[
[u] \otimes_A [\D] 
= \frac12 \Sig\big( L^\od_{\kappa,\rho}(G,\D) \big) .
\]
\end{thm}
\begin{remark}
This recovers the main result of \cite{LS17}, without assuming $\D$ to be invertible. In contrast with \cref{thm:even_pairing_spec_loc}, there is no correction term for non-invertible $\D$, due to the simple fact that the matrix $\mattwo{0}{1}{1}{0}$ has vanishing signature.
\end{remark}
\begin{proof}
The proof is similar to the proof of \cref{thm:even_pairing_spec_loc}, so we shall be somewhat brief. 
By \cref{thm:odd_spec_loc_rho_inv}, the signature of $L^\od_{\kappa,\rho}(G,\D)$ is independent of the choices of $\kappa$ and $\rho$, so we may assume $\rho$ to be as large as necessary. 
From \cref{thm:odd_index_pairing_spec_loc,lem:odd_spec_loc_oplus} we obtain the direct sum decomposition 
\begin{align*}
[u] \otimes_A [\D] 
&= \SF\big( L^\od_{\kappa,\rho}(1,\D) \sfarrow L^\od_{\kappa,\rho}(G,\D) \big) + \SF\big( L^\od_{\kappa,\rho^c}(1,\D) \sfarrow L^\od_{\kappa,\rho^c}(G,\D) \big) .
\end{align*}
Choosing $\rho$ to be sufficiently large such that $\max\big\{1,\|G\|\big\} < \kappa \rho$, the second summand vanishes and we obtain 
\[
[u] \otimes_A [\D] 
= \SF\big( L^\od_{\kappa,\rho}(1,\D) \sfarrow L^\od_{\kappa,\rho}(G,\D) \big) 
= \frac12 \Big( \Sig\big( L^\od_{\kappa,\rho}(G,\D) \big) - \Sig\big( L^\od_{\kappa,\rho}(1,\D) \big) \Big) .
\]
Finally, we note that $L^\od_{\kappa,\rho}(1,\D)$ is invertible for all $\kappa$, and therefore its signature is independent of $\kappa$. The statement then follows from 
\[
\Sig\big( L^\od_{\kappa,\rho}(1,\D) \big) 
= \Sig\left( \mattwo{0}{1}{1}{0} \right) 
= 0 .
\qedhere
\]
\end{proof}

\appendix 

\section{Fredholm operators and \K-theory}
\label{app:Fredholm}

Let $E$ be a Hilbert $C^*$-module over a $C^*$-algebra $A$. We denote by $\mL_A(E)$ the $C^*$-algebra of adjointable operators on $E$, and by $\mK_A(E)$ the $C^*$-subalgebra of compact ($=$ `generalised compact' or $A$-compact) operators. In this Appendix, we will consider Fredholm operators on $E$ and in particular their index and spectral flow taking values in the $\K$-theory of $A$. Some computations make use of Kasparov's bivariant $\KK$-theory and the Kasparov product \cite{Kas80b}, and we refer to the book \cite{Blackadar98} for further details on this theory. 

A regular operator $\D$ on $E$ is called \emph{Fredholm} if there exist a \emph{left parametrix} $Q_l$ and a \emph{right parametrix} $Q_r$ such that (the closure of) $Q_l\D - 1$ and $\D Q_r - 1$ are compact endomorphisms on $E$. 
If $\D$ is regular and Fredholm, we denote by $\Index(\D) \in \K_0(A)$ its Fredholm index (for its construction, we refer to \cite[\S2.2]{vdD19_Index_DS} and references therein). 

Let now $\D$ be a regular \emph{self-adjoint} Fredholm operator on $E$. We consider both the $\Z_2$-graded case ($j=0$) and the ungraded case ($j=1$). In the graded case ($j=0$), we assume that $\D$ is odd with respect to the direct sum decomposition $E = E_+ \oplus E_-$ given by the $\Z_2$-grading, so that we may write 
\[
\D = \mattwo{0}{\D_-}{\D_+}{0} .
\]
By \cite[Proposition 2.14]{vdD19_Index_DS}, $\D$ yields a well-defined class $[\D]$ in $\KK^0(\C,A) \simeq \K_0(A)$ resp.\ $\KK^1(\C,A) \simeq \K_1(A)$ (see \cite[\S2.2]{vdD19_Index_DS} for the construction of this class). Furthermore, if two such operators $\D$ and $\D'$ are homotopic, then $[\D]=[\D']$. 
In the graded case ($j=0$), the class $[\D] \in \KK^0(\C,A)$ corresponds to $\Index(\D_+) \in \K_0(A)$ under the standard isomorphism $\KK^0(\C,A) \simeq \K_0(A)$. 

We recall from \cite[Proposition A.11]{vdD25_Callias_25a} the stability under relatively compact perturbations: 
if $R$ is a symmetric operator on $E$ which is relatively $\D$-compact (i.e., $R(\D\pm i)^{-1}$ is compact), then $\D+R$ is also regular, self-adjoint, and Fredholm, and $[\D+R] = [\D] \in \KK^j(\C,A)$ (where $j=0$ if $R,\D$ are odd, and $j=1$ otherwise).

\subsection{Relative index and spectral flow}

We briefly recall the definitions and basic properties of the (even or odd) relative index and the (even or odd) spectral flow on a Hilbert $A$-module $E$, taking values in the \K-theory group $\K_j(A)$. For a more detailed exposition, we refer to \cite[Appendix A]{vdD25_DS_ind_sf_25b} and \cite[\S3 \& \S8]{Wah07}. 

Consider two projections $P,Q \in \mL_A(E)$, such that the difference $P-Q$ is a \emph{compact} endomorphism on $E$. 
In the ungraded case ($j=0$), we note that compactness of $P-Q$ implies that $Q \colon \Ran(P) \to \Ran(Q)$ is a Fredholm operator and thus has a $\K_0(A)$-valued index. 
In the graded case ($j=1$), we assume $E = E_+ \oplus E_-$ is $\Z_2$-graded, and we require in addition that $2P-1$ and $2Q-1$ are odd. 
In this case, we can write 
\begin{equation}
\label{eq:proj_odd}
P = \frac12 \mattwo{1}{U_P^*}{U_P}{1} ,
\quad\text{and}\quad
Q = \frac12 \mattwo{1}{U_Q^*}{U_Q}{1} ,
\end{equation}
where $U_P,U_Q\colon E_+ \to E_-$ are unitaries such that $U_P U_Q^*$ lies in the minimal unitisation of the compact operators on $E_-$. 

\begin{defn}
\label{defn:rel-ind}
Consider projections $P,Q \in \mL_A(E)$ with $P-Q \in \mK_A(E)$. 
We define the (even or odd) relative index $\relind_j(P,Q) \in \K_j(A)$ as follows. 
In the ungraded case $j=0$, we define the \emph{(even) relative index of $(P,Q)$} by 
\[
\relind(P,Q) \equiv \relind_0(P,Q) := \Index \big( Q \colon \Ran(P) \to \Ran(Q) \big) \in \K_0(A) .
\]
In the graded case $j=1$, we additionally require $2P-1$ and $2Q-1$ to be odd, and define the \emph{odd relative index of $(P,Q)$} by 
\[
\relind_1(P,Q) := \left[ \mattwo{1}{0}{0}{U_P U_Q^*} \right] \in \K_1(A) ,
\]
where $U_P$ and $U_Q$ are obtained from $P$ and $Q$ as in \cref{eq:proj_odd}.
\end{defn}

Now consider a regular self-adjoint operator $\D(\cdot)$ on the Hilbert $C([0,1],A)$-module $C([0,1],E)$ corresponding to a family of operators $\{\D(t)\}_{t\in[0,1]}$ on $E$. 
We can define its (even or odd) spectral flow, provided that there exist locally trivialising families (see \cite[Definition 2.6]{vdD25_Callias_25a}). Note that, in our main case of interest, such trivialising families always exist (see \cref{prop:sf_rel_cpt_family} below). 

Given a regular self-adjoint Fredholm operator $\D$ on $E$ and two trivialising operators $\B_0$ and $\B_1$ for $\D$, we note that $P_{>0}(\D+\B_1) - P_{>0}(\D+\B_0)$ is compact (and if $j=1$, the operators $2P_{>0}(\D+\B_0)-1$ and $2P_{>0}(\D+\B_1)-1$ are odd), and we introduce the notation
\begin{equation}
\label{eq:ind}
\ind_j(\D,\B_0,\B_1) := \relind_j\big( P_{>0}(\D+\B_1) , P_{>0}(\D+\B_0) \big) .
\end{equation}
The spectral flow is then defined as follows. 

\begin{defn}
\label{defn:spectral_flow}
Let $j=0$ or $j=1$. 
Let $\D(\cdot) = \{\D(t)\}_{t\in[0,1]}$ be a regular self-adjoint operator on the Hilbert $C([0,1],A)$-module $C([0,1],E)$, for which locally trivialising families exist. (If $j=1$, we require $E$ to be $\Z_2$-graded and $\D$ (as well as all trivialising families) to be odd.) 
Let $0 = t_0 < t_1 < \ldots < t_n = 1$ be such that there is a trivialising family $\{\B^i(t)\}_{t\in[t_i,t_{i+1}]}$ of $\{\D(t)\}_{t\in[t_i,t_{i+1}]}$ for each $i=0,\ldots,n-1$. 
Assume that the endpoints $\D(0)$ and $\D(1)$ are invertible. 
Then we define 
\begin{multline*}
\SF_j\big(\{\D(t)\}_{t\in[0,1]} \big) 
:= \ind_j\big(\D(0),0,\B^0(0)\big) + \sum_{i=1}^{n-1} \ind_j\big(\D(t_i),\B^{i-1}(t_i),\B^i(t_i)\big) \\
+ \ind_j\big(\D(1),\B^{n-1}(1),0\big) 
\in K_j(A) ,
\end{multline*}
where $\ind_j$ is defined in \cref{eq:ind}. 
We call $\SF_0$ the (even) \emph{spectral flow} and $\SF_1$ the \emph{odd spectral flow} of the family $\{\D(t)\}_{t\in[0,1]}$. 
We will often simply write $\SF \equiv \SF_0$ for the (even) spectral flow. 
The definition of the spectral flow is independent of the choice of subdivision and the choice of trivialising families $\{\B^i(t)\}_{t\in[t_i,t_{i+1}]}$. 
\end{defn}

\begin{prop}[{\cite[Proposition 2.8]{vdD25_Callias_25a} \& \cite[Proposition A.11]{vdD25_DS_ind_sf_25b}}]
\label{prop:sf_rel_cpt_family}
Let $\D(\cdot) = \{\D(t)\}_{t\in[0,1]}$ be a regular self-adjoint operator on the Hilbert $C([0,1],A)$-module $C([0,1],E)$. 
(In the graded case $j=1$, $E$ is $\Z_2$-graded and $\D$ is odd.) 
Assume that
the endpoints $\D(0)$ and $\D(1)$ are invertible, 
$\D(t) \colon \Dom\D(0) \to E$ depends norm-continuously on $t$, and 
$\D(t)-\D(0)$ is relatively $\D(0)$-compact for each $t\in[0,1]$. 
Then there exists a trivialising family for $\{\D(t)\}_{t\in[0,1]}$ and 
\begin{align*}
\SF_j\big(\{\D(t)\}_{t\in[0,1]} \big) 
&= \relind_j\big( P_{>0}(\D(1)) , P_{>0}(\D(0)) \big) 
\in \K_j(A) .
\end{align*}
\end{prop}

In the setting of \cref{prop:sf_rel_cpt_family}, it follows in particular that the spectral flow depends only on the endpoints $\D(0)$ and $\D(1)$. We then introduce the notation $\SF_j \big( \D(0) \sfarrow \D(1) \big)$ for the spectral flow of the straight line path from $\D(0)$ to $\D(1)$:
\begin{equation}
\label{eq:SF_endpoints}
\SF_j \big( \D(0) \sfarrow \D(1) \big) := \SF_j \big( \big\{ (1-t) \D(0) + t \D(1) \big\}_{t\in[0,1]} \big) .
\end{equation}

\begin{coro}[{\cite[Corollary A.12]{vdD25_DS_ind_sf_25b}}]
\label{coro:sf_index_PuP}
Let $\D$ be an invertible regular self-adjoint operator on the Hilbert $A$-module $E$. 
Consider a unitary operator $u \in \mL_A(E)$ such that $u\colon\Dom\D\to\Dom\D$ and $[\D,u]$ is relatively $\D$-compact. 
Let $\chi \colon [0,1] \to \R$ be any continuous function satisfying $\chi(0) = 0$ and $\chi(1) = 1$. 
For $t\in[0,1]$ we define $\D(t) := (1-\chi(t)) \D + \chi(t) u^* \D u = \D + \chi(t) u^* [\D,u]$. 
Then we have the equality 
\[
\SF_0 \big( \D \sfarrow u^* \D u \big)
= \SF_0\big(\{\D(t)\}_{t\in[0,1]} \big) 
= - \Index \big( P_{>0}(\D) u P_{>0}(\D) \big) 
\in \K_0(A) .
\]
\end{coro}

The following result shows that the (even or odd) spectral flow implements the Bott periodicity isomorphism $\K_{j+1}\big(C_0(\R,A)\big) \xrightarrow{\simeq} \K_j(A)$, which is given by taking the Kasparov product with the \K-homology class $[-i\partial_t] \in \KK^1\big( C_0(\R) , \C \big)$. The latter is represented by the spectral triple $\big( C_c^\infty(\R) , L^2(\R) , -i\partial_t \big)$. 

\begin{prop}[{\cite[Proposition A.13]{vdD25_DS_ind_sf_25b}, cf.\ \cite[\S4 \& \S8]{Wah07}}]
\label{prop:SF_Kasp_prod}
Consider a regular self-adjoint Fredholm operator $\D(\cdot) = \{\D(t)\}_{t\in[0,1]}$ on the Hilbert $C([0,1],A)$-module $C([0,1],E)$, with $\D(0)$ and $\D(1)$ invertible. 
(In the graded case $j=1$, $E$ is $\Z_2$-graded and each $\D(t)$ is odd.) 
We extend the family to $\R$ by setting $\D(t) := \D(0)$ for all $t<0$ and $\D(t) := \D(1)$ for all $t>1$, and we view $\D(\cdot)$ as a regular self-adjoint Fredholm operator on the Hilbert $C_0(\R,A)$-module $C_0(\R,E)$, defining a class $[\D(\cdot)] \in \K_{j+1}\big(C_0(\R,A)\big)$. 

If there exist locally trivialising families for $\{\D(t)\}_{t\in\R}$, then
\[
\SF_j\big(\{\D(t)\}_{t\in[0,1]}\big) = [\D(\cdot)] \otimes_{C_0(\R)} [-i\partial_t] \in K_j(A) .
\]
\end{prop}

\section{Computation of a Kasparov product}
\label{app:Kasp_prod}

In this section, we will compute a particular case of a Kasparov product in the unbounded picture \cite{BJ83} of Kasparov's bivariant $\KK$-theory \cite{Kas80b}. 
The computation makes use of the description of the unbounded Kasparov product given by Lesch--Mesland \cite[Theorem 7.4]{LM19}. 
The (infinite volume) spectral localiser naturally appears from the construction of this Kasparov product, and this observation is the crucial ingredient in the results of \cref{sec:index_pairing}. 

Consider unital $C^*$-algebras $A$ and $B$. 
Let $\A \subset A$ be a dense unital $*$-subalgebra, and consider an (even) unbounded Kasparov $A$-$B$-module $(\A,E,\D)$ over $A$ representing a \KK-theory class $[\D] \in \KK^0(A,B)$. 
We assume that the corresponding representation $\pi \colon A \to \mL_B(E)$ is \emph{nondegenerate}, i.e., $\pi(A)\cdot E$ is dense in $E$. 
On the $\Z_2$-graded Hilbert $B$-module $E = E_+ \oplus E_-$ we may write the operator $\D$ and the $\Z_2$-grading $\Gamma_\D$ as 
\[
\D = \mattwo{0}{\D_-}{\D_+}{0} , 
\qquad 
\Gamma_\D = \mattwo{1}{0}{0}{-1} .
\]
Consider also a norm-continuous family $\{S(t)\}_{t\in\R} \subset M_{2n}(A)$ with the following properties:
\begin{itemize}
\item $S(t) = S(-1)$ for all $t\leq-1$ and $S(t) = S(1)$ for all $t\geq1$;
\item $S(-1)$ and $S(1)$ are invertible;
\item $S(t)$ is odd w.r.t.\ the $\Z_2$-grading $\Gamma_S = 1\oplus(-1)$ on $A^{\oplus2n} = A^{\oplus n} \oplus A^{\oplus n}$;
\item $\pi(S(t))$ preserves $\Dom(\D)^{\oplus2n}$, the commutator $[\D,\pi(S(t))]$ is bounded, and $t \mapsto [\D,\pi(S(t))]$ is norm-continuous. 
\end{itemize}
We then obtain an odd self-adjoint Fredholm operator 
\[
S \in C_0\big( \R , M_{2n}(A) \big)
\subset \mL_{C_0(\R,A)}\big(C_0(\R,A^{\oplus2n})\big) ,
\]
which represents a \K-theory class $[S] \in \K_0\big(C_0(\R,A)\big)$ (see \cref{app:Fredholm}). 

We consider a real parameter $\kappa$ satisfying 
\begin{equation}
\label{eq:kappa_Kasp_prod}
0 < \kappa < \min \Big\{ \big\|S(-1)^{-1}\big\|^{-2} \big\|[\D,S(-1)]\big\|^{-1} , \big\|S(1)^{-1}\big\|^{-2} \big\|[\D,S(1)]\big\|^{-1} \Big\} .
\end{equation}

\begin{prop}
\label{prop:00_Kasp_prod_S_D}
The pairing of $[S] \in \K_0\big(C_0(\R)\otimes A\big)$ with $[\D]\in\KK^0(A,B)$ over $A$ is given by 
\[
[S] \otimes_A [\D] 
= [\mL_\kappa] \in \K_0\big( C_0(\R)\otimes B \big) ,
\]
where the self-adjoint Fredholm operator $\mL_\kappa$ on the Hilbert $C_0(\R,B)$-module $C_0(\R,E^{\oplus2n})$ is given by 
\(
\mL_\kappa(t) := \kappa \Gamma_S \D + S(t) ,
\)
and where the parameter $\kappa$ satisfies \cref{eq:kappa_Kasp_prod}. 
\end{prop}
\begin{proof}
The operator $\mL_\kappa$ is given by a family $\{\mL_\kappa(t)\}_{t\in\R}$ of self-adjoint operators with constant domain $\Dom \mL_\kappa(t) = (\Dom\D)^{\oplus2n}$, such that $\mL_\kappa(t) \in \mL_B\big((\Dom\D)^{\oplus2n},E^{\oplus2n}\big)$ depends norm-continuously on $t$. Hence $\mL_\kappa$ defines a regular self-adjoint operator on the Hilbert $C_0(\R,B)$-module $C_0(\R,E^{\oplus2n})$, which is odd with respect to the $\Z_2$-grading $\Gamma_S \Gamma_\D$. 
Since $\D$ has compact resolvents and $\mL_\kappa(t) - \kappa \Gamma_S \D$ is bounded, we see that 
$\mL_\kappa(t)$ also has compact resolvents, 
and in particular each $\mL_\kappa(t)$ is Fredholm. 
Moreover, the operators $\mL_\kappa(\pm1) = \kappa \Gamma_S \D + S(\pm1)$ are invertible by the same argument as in \cref{lem:even_spec_loc_inv,lem:odd_spec_loc_inv} (using the assumption \eqref{eq:kappa_Kasp_prod}). 
We then obtain a (left and right) parametrix $Q$ for $\mL_\kappa$ given by $Q(t) = (1-\chi(t)) \mL_\kappa(t)^{-1} + \chi(t) (\mL_\kappa(t)+i)^{-1}$, for any $\chi \in C_c(\R)$ such that $\chi(t) = 1$ for all $t\in[-1,1]$. 
Thus $\mL_\kappa$ is Fredholm and therefore yields a well-defined class $[\mL_\kappa] \in \K_0\big(C_0(\R,B)\big)$. 

We may choose a smooth function $\varphi \colon \R \to [1,\infty)$ such that $\varphi(t) = 1$ for all $t\in[-1,1]$ and $\lim_{|t|\to\infty}\varphi(t) = \infty$. 
Then the operator $S'$ given by $S'(t) := \varphi(t) S(t)$ has compact resolvents and therefore defines an unbounded Kasparov $\C$-$C_0(\R)\otimes A$-module $\big( \C , C_0(\R,A^{\oplus2n}) , S' \big)$, such that $[S'] = [S] \in \K_0\big( C_0(\R)\otimes A \big)$. 
Indeed, the norm-continuous family $[0,1]\times\R \ni (s,t) \mapsto \varphi(t)^s S(t) \equiv S_s(t)$ yields a regular self-adjoint operator $S_\bullet$ on $C_0\big([0,1]\times\R,A^{\oplus2n}\big)$, which is Fredholm with a parametrix given by $(s,t) \mapsto (1-\chi(s,t)) S_s(t)^{-1}$, for any $\chi \in C_c\big([0,1]\times\R\big)$ with $\chi(s,t) = 1$ for all $(s,t) \in [0,1] \times [-1,1]$. Thus $S_\bullet$ provides a homotopy between $S$ and $S'$, so that $[S] = [S']$. 

Similarly, we may replace $\mL_\kappa$ by the operator $\mL'_\kappa := \kappa \Gamma_S \D + S'$. 
Indeed, since $[\D,S]$ is bounded, the graded commutator $[\Gamma_S \D,S']$ is relatively bounded by $S'$, so that $\mL'_\kappa$ is regular and self-adjoint on the domain $\Dom(\D) \cap \Dom(S')$ (cf.\ \cite[Theorem 7.10]{KL12}). 
Since $S'$ has compact resolvents, it then follows that also $\mL'_\kappa$ has compact resolvents and therefore defines an unbounded Kasparov $\C$-$C_0(\R,B)$-module $\big( \C , C_0(\R,E^{\oplus2n}) , \mL'_\kappa \big)$. 
The regular self-adjoint operator $\mL_{\kappa,\bullet}$ given by $\mL_{\kappa,s}(t) := \kappa \Gamma_S \D + S_s(t)$ has a parametrix given by $(s,t) \mapsto (1-\chi(s,t)) \mL_{\kappa,s}(t)^{-1} + \chi(s,t) (\mL_{\kappa,s}(t)+i)^{-1}$, so it provides a homotopy between $\mL_\kappa$ and $\mL'_\kappa$ and we therefore have that $[\mL'_\kappa] = [\mL_\kappa] \in \K_0\big( C_0(\R,B) \big)$. 

It remains to show that $\mL'_\kappa$ represents the (unbounded) Kasparov product of $S'$ with $\D$. 
We will compute the Kasparov product by applying \cite[Theorem 7.4]{LM19} to the operators $S'$ and $T := \kappa \Gamma_S \D$ on $C_0\big(\R,E^{\oplus2n}\big)$.
Note that $A$ is unital, and there is a dense $*$-subalgebra $\A\subset A$ such that $[\D,a]$ is bounded for all $a\in\A$. Since $\D$ acts diagonally on $E^{\oplus2n}$, condition (i) of \cite[Theorem 7.4]{LM19} is satisfied by the dense submodule $C_c\big(\R,\A^{\oplus2n}\big)$. 
Condition (ii) is trivially satisfied ($\C$ commutes with $T$). 
Finally, condition (iii) is satisfied because the commutator $[\D,S]$ is bounded, so that the anticommutator $[T,S']_+$ is relatively bounded by $S'$. 
Hence the statement of \cite[Theorem 7.4]{LM19} shows that $S' + T = \varphi S + \kappa \Gamma_S \D = \mL'_\kappa$ indeed represents the Kasparov product of $[S']=[S]$ with $[\kappa\D] = [\D]$. 
\end{proof}

\begin{coro}
\label{coro:10_Kasp_prod_S_D}
In the setting of \cref{prop:00_Kasp_prod_S_D}, consider now instead a \emph{trivially graded} element $S \in C_0\big( \R, M_n(A) \big)$. 
The pairing of $[S] \in \K_1\big(C_0(\R)\otimes A\big)$ with $[\D]\in\KK^0(A,B)$ over $A$ is then given by 
\[
[S] \otimes_A [\D] 
= [\mL^\ev_\kappa] \in \K_1\big( C_0(\R)\otimes B \big) ,
\]
where the self-adjoint Fredholm operator $\mL^\ev_\kappa$ on the Hilbert $C_0(\R,B)$-module $C_0(\R,E^{\oplus n})$ is given by 
\(
\mL^\ev_\kappa(t) := \kappa \D + \Gamma_\D S(t) ,
\)
and where the parameter $\kappa$ satisfies \cref{eq:kappa_Kasp_prod}. 
\end{coro}
\begin{proof}
The statement follows by combining \cref{prop:00_Kasp_prod_S_D} with the description of the odd-even Kasparov product given in \cite[Example 2.38]{BMS16}. 
\end{proof}

\begin{coro}
\label{coro:01_Kasp_prod_S_D}
In the setting of \cref{prop:00_Kasp_prod_S_D}, consider now instead an \emph{odd} unbounded Kasparov $A$-$B$-module $(\A,E,\D)$. 
The pairing of $[S] \in \K_0\big(C_0(\R)\otimes A\big)$ with $[\D]\in\KK^1(A,B)$ over $A$ is then given by 
\[
[S] \otimes_A [\D] 
= [\mL^\od_\kappa] \in \K_1\big( C_0(\R)\otimes B \big) ,
\]
where the self-adjoint Fredholm operator $\mL^\od_\kappa$ on the Hilbert $C_0(\R,B)$-module $C_0(\R,E^{\oplus2n})$ is given by 
\(
\mL^\od_\kappa(t) := \kappa \Gamma_S \D + S(t) ,
\)
and where the parameter $\kappa$ satisfies \cref{eq:kappa_Kasp_prod}. 
\end{coro}
\begin{proof}
The statement follows by combining \cref{prop:00_Kasp_prod_S_D} with the description of the even-odd Kasparov product given in \cite[Example 2.37]{BMS16}. 
\end{proof}

\small 


\providecommand{\noopsort}[1]{}\providecommand{\Dunsort}{}\providecommand{\DMsort}{}
\providecommand{\bysame}{\leavevmode\hbox to3em{\hrulefill}\thinspace}
\providecommand{\MR}{\relax\ifhmode\unskip\space\fi MR }
\providecommand{\MRhref}[2]{%
  \href{http://www.ams.org/mathscinet-getitem?mr=#1}{#2}
}
\providecommand{\href}[2]{#2}
\providecommand{\doi}[1]{\href{https://doi.org/#1}{doi:#1}}
\providecommand{\doilinktitle}[2]{#1}
\providecommand{\doilinkbooktitle}[2]{\href{https://doi.org/#2}{#1}}
\providecommand{\doilinkjournal}[2]{\href{https://doi.org/#2}{#1}}
\providecommand{\doilinkvynp}[2]{\href{https://doi.org/#2}{#1}}
\providecommand{\eprint}[2]{#1:\href{https://arxiv.org/abs/#2}{#2}}

\end{document}